\documentclass{article}
\usepackage[a4paper, portrait, margin=1in]{geometry}
\usepackage{authblk}
\usepackage[english]{babel}
\usepackage[utf8]{inputenc}
\usepackage[T1]{fontenc}
\usepackage{amsmath, amsthm, amssymb,amsfonts}
\usepackage{mathrsfs}
\usepackage{physics}
\usepackage{array}
\usepackage{helvet}
\usepackage{cite}
\usepackage{etoolbox}
\usepackage{graphicx}
\usepackage{titlesec}
\usepackage{booktabs}
\usepackage{xcolor} 
\usepackage{scrextend}
\usepackage{fancyhdr}
\usepackage{graphicx}
\usepackage[colorlinks, citecolor=red, urlcolor=blue]{hyperref}

\newtheorem{lemma}{\scshape Lemma}
\newtheorem{proposition}{\scshape Proposition}

\newtheorem{remark}{\scshape Remark}
\newtheorem{theorem}{\scshape Theorem}

\title{Boundary Stabilization of a Degenerate Euler-Bernoulli Beam under Axial Force and Time Delay}

\author[1]{Ben Bakary Junior SIRIKI \footnote{Corresponding author: \url{benbjsiriki@gmail.com}}}
\author[2]{Adama COULIBALY}

\affil[1]{Université Nangui Abrogoua, Abidjan, Côte D'Ivoire}
\affil[2]{Université Félix Houphouët-Boigny, Abidjan, Côte D'Ivoire}

\date{\today}

\providecommand{\keywords}[1]
{
	\small
	\textbf{Keywords:} #1
}

\begin{document}
	
\maketitle
	
	\begin{abstract}
This paper provides a qualitative analysis of a non-uniform Euler-Bernoulli beam with degenerate flexural rigidity, subjected to axial force and boundary control with time delay $\tau > 0$. By reformulating the system as an abstract evolution problem in an augmented Hilbert space incorporating weighted Sobolev spaces, we employ semigroup theory to ensure well-posedness. Using the energy multiplier method and a non-standard Lyapunov functional featuring weighted integral terms, we establish uniform exponential energy decay and provide a precise decay rate estimate. This work extends the results of Salhi et al. \cite{salhi2025} and Siriki et al. \cite{siriki2025} by incorporating axial force and generalized control laws, including rotational velocity control. 
The proposed framework offers a robust approach for analyzing complex distributed systems.
	\end{abstract}
	
	\keywords{Euler-Bernoulli beam, degenerate flexural rigidity, axial force, Boundary delay control, exponential stability, weighted Sobolev spaces}

\section{Introduction}
	The analysis and control of systems governed by partial differential equations (PDEs), commonly referred to as distributed parameter systems, constitute a highly active field of investigation within the scientific community. Specifically, the development of complex physical structures in various engineering fields, including aeronautics \cite{chakravarthy2010, vinod2007, kundu2017}, robotics \cite{kundu2017, luo1993}, and civil engineering \cite{du2022, fryba1996, li2013}, has ensured that the control and stabilization of Euler-Bernoulli beams remains an area of sustained research focus for several decades.

In this article, we investigate the well-posedness and exponential stability of a non-uniform Euler-Bernoulli beam problem. This beam is clamped at one end, and its free end is subjected to a linear control system incorporating a fixed time delay.
The vibrational motion of the beam is governed by the following equations:
\begin{equation} \label{eq:1}
	\left\{
	\begin{array}{>{\displaystyle}l}
	 u_{tt} + \left( \sigma(x) u_{xx} \right)_{xx} - \left( q(x) u_x \right)_x = 0 , \quad (x, t) \in (0,1) \times (0, \infty), \\
	u(0,t) = \mathcal{B} u(0,t) = 0, \quad t \in (0, \infty),  \\
	- \sigma(1) u_{xx}(1,t) = \kappa_r u_x(1,t) + \kappa_a u_{xt}(1,t), \quad t \in (0, \infty),  \\
	\left( \sigma (x) u_{xx} \right)_x(1,t) - q(1) u_x(1,t) = \kappa_v u_t(1,t) + \kappa_d  u_t(1,t-\tau) + \kappa_b u(1,t), \quad t \in (0, \infty), \\
	u(x,0) = u_0(x), \, u_t(x,0) = u_1(x), \quad x \in [0, 1], \\
	u_t(1, t-\tau) = g_0(t-\tau), \quad t \in (0, \tau),
	\end{array}
	\right.
\end{equation}
where $\tau > 0$ denotes the time delay and $u(x,t)$ is the deflection of the beam at position $x$ and time $t$. The functions $\sigma: (0,1) \rightarrow \mathbb{R}_+$ and $q: (0,1) \rightarrow \mathbb{R}_+$ denote, respectively, the flexural rigidity and the axial force distribution along the beam. We assume that the flexural rigidity $\sigma(x)$ is degenerate, and its degeneracy is measured by the constant $K_\sigma$ defined by:
\begin{align} \label{eq:2}
	K_\sigma := \underset{x \in (0, \, 1]}{\sup} \frac{x |\sigma'(x)|}{\sigma(x)}.
\end{align}
In particular, $\sigma$ is said to be:
\begin{itemize}
\item weakly degenerate (WD) if $\sigma(0) = 0, \; \sigma > 0$ on $(0, 1]$, $\sigma \in C[0,1] \cap C^1(0,1]$ and if $K_\sigma \in (0, 1)$;
\item strongly degenerate (SD) if $\sigma(0) = 0, \; \sigma > 0$ on $(0, 1]$, $\sigma \in C^1[0,1]$ and if $K_\sigma \in [1, 2)$.
\end{itemize}
Due to the degeneracy of the problem at $x=0$, we point out that the operator $\mathcal{B}$ is defined by:
\begin{gather} \label{eq:3}
\mathcal{B} u(x,t) := \left\{
\begin{array}{l}
u_x(x,t) \quad \text{ if } \sigma \text{ is (WD),} \\ ~ \\
\left(\sigma u_{xx}\right)(x,t) \quad \text{ if } \sigma \text{ is (SD).}
\end{array}
\right.
\end{gather}
In addition, the constants $\kappa_r, \kappa_v, \kappa_b \geq 0$, $\kappa_a > 0$ and $\kappa_d \neq 0$ are the control gains and $u_0, \, u_1, \, g_0$ are the initial data.

The qualitative analysis of Euler-Bernoulli beam models subjected to linear boundary controls is a widely investigated topic in the literature.
Research interest was initially spurred in the 1980s by the NASA SCOLE (Spacecraft Control Laboratory Experiment) project, which aimed at controlling the dynamics of large flexible space structures (see \cite{balakrishnan1986}).
Subsequently, since the early 2000s, the stabilization of these systems has been the subject of extensive research, including both uniform beam models and those with variable physical parameters.
Research in this area has primarily converged upon two distinct stabilization strategies.
For instance, Bao-Zhu Guo et al. in \cite{guo2001} studied a uniform beam model and demonstrated its well-posedness and exponential stability using a frequency domain approach. They subsequently extended these results to a model with variable coefficients (see \cite{guo2002}). These analyses fundamentally rely on the associated linear operator being dissipative and discrete, and on the real part of the asymptotic expression of its generalized eigenvalues being negative.
Another common strategy relies on the Lyapunov method. Dadfarnia et al. \cite{dadfarnia2004} developed a feedback controller to exponentially stabilize the displacement of a cantilever beam. Their approach is based on the construction of a suitable Lyapunov functional $V(t)$ whose time derivative satisfies the differential inequality $\dfrac{d V(t)}{dt} \leqslant - \lambda V(t)$, with $\lambda > 0$.
Other studies have addressed questions of solution existence and stability for problems similar to \eqref{eq:1} involving different boundary controls with or without delay (see \cite{camasta2025a, camasta2025, camasta2024, guo2008, marc2017, teya2023, wang2005, shang2012}), concerning degenerate flexural rigidity (see \cite{salhi2025, akil2025}), or considering the presence of an axial force (see \cite{siriki2025, ledkim2023, benamara2025}).
It is noteworthy that the overwhelming majority of existing work is limited to non-degenerate beam models and/or does not incorporate the simultaneous effect of axial force and time delay.
Consequently, questions regarding the existence and stability of solutions in a context integrating the degeneracy of flexural rigidity, the presence of an axial force, and a time delay within a feedback loop remain largely unexplored. 

This paper addresses this gap by developing an analytical framework that accounts for these physical complexities simultaneously.
We obtain two main results under the following condition:
\begin{align} \label{eq:4}
	|\kappa_d| < \kappa_v. 
\end{align}
First, we prove the well-posedness of problem \eqref{eq:1}. To overcome the loss of uniform ellipticity due to the degenerate flexural rigidity $\sigma(x)$ and the impact of the time delay $\tau > 0$, we construct a novel Hilbert state space built upon suitable weighted Sobolev spaces (see \cite{salhi2025, alabau2006, alabau2017}). On this state space, we reformulate system \eqref{eq:1} as an abstract evolution problem. 
The application of semigroup theory then ensures the existence and uniqueness of a solution which depends continuously on the initial data.
Second, we establish the exponential stability of the solution of system \eqref{eq:1}. The combination of degeneracy, axial force, and time delay in system \eqref{eq:1} can potentially lead to instability (see \cite{datko1991, datko1993}). To address this, we construct a novel Lyapunov functional from weighted integral terms. By performing rigorous estimates of this functional and its time derivative via the energy multiplier method, we show that the system's energy decays exponentially to zero, and provide a precise estimate of the decay rate.

The remainder of this article is organized as follows. Section 2 introduces the assumptions on the axial force function and defines the weighted functional spaces adapted to the degenerate flexural rigidity.
The well-posedness of problem \eqref{eq:1} is then addressed in Section 3, where we construct the energy space and ensure the existence and uniqueness of the solution.
In Section 4, we prove the exponential stability of system \eqref{eq:1} and derive a precise estimate of the energy decay rate. 
Finally, the paper closes with a conclusion.
 
\section{Assumptions and preliminaries}

	\subsection{Assumptions}
Throughout the remainder of this work, we assume that the axial force distribution $q$ satisfies the following structural conditions:
\begin{equation} \label{eq:5}
	\left\{
	\begin{array}{>{\displaystyle}l}
	 q \in W^{1, \infty}(0,1), \\
	 0 < q_0 \leqslant q(x) \leqslant q_1, \; \forall x \in [0, 1],\\
	 |q'(x)| \leqslant q_2, \; \forall x \in [0, 1].
	\end{array}
	\right.
\end{equation}
Thus, we have the following immediate functional implications:
\begin{enumerate}
\item since $q \in W^{1, \infty}(0,1)$, then $\sqrt{q} u \in L^2(0,1)$ for all $u \in L^2(0,1)$;
\item for stabilization issue, we have $\displaystyle \frac{x |q'(x)|}{q(x)} \leqslant \frac{q_2}{q_0}$ for all $x \in (0,1)$.
\end{enumerate}
	
	\subsection{Functional spaces}

In order to construct the Hilbert state space on which we will  examine the well-posedness of the delayed problem \eqref{eq:1}, and consistent with the methodology employed in the anterior works \cite{salhi2025, siriki2025, camasta2024, camasta2025}, we introduce some weighted Sobolev spaces naturally associated with degenerate operators. \newline
Let the Hilbert space:
\begin{equation} \label{eq:6}
	V_\sigma^2(0, 1) := \left\{
	\begin{array}{>{\displaystyle}l}
	 \Big\{ u \in H^1(0,1): \, u' \text{ is absolutely continuous on } [0,1],  \\
	  \qquad \qquad \qquad \quad \sqrt{\sigma} u'' \in L^2(0,1) \; \text{ if } \underline{\sigma \text{ is (WD)}} \Big\};\\
	 \Big\{ u \in H^1(0,1): \, u' \text{ is locally absolutely continuous on } (0,1], \\
	 \qquad \qquad \qquad \qquad \qquad \qquad \sqrt{\sigma} u'' \in L^2(0,1) \; \text{ if } \underline{\sigma \text{ is (SD)}} \Big\}.
	\end{array}
	\right.
\end{equation}
equipped with the inner product defined as:
\begin{gather} \label{eq:7}
\prec u, v \succ_{\, 2,\sigma} = \int_{0}^{1} \bigg( \sigma(x) u''(x) v''(x) \, + u'(x) v'(x) + \, u(x) v(x) \bigg) \, dx 
\end{gather}
for all $u, \, v \in V_\sigma^2(0, 1)$ and associated norm:
\begin{align} \label{eq:8}
\lVert u \rVert_{2, \sigma} & := \left( \lVert \sqrt{\sigma} u'' \rVert_{L^2(0,1)}^2 + \lVert u' \rVert_{L^2(0,1)}^2 + \lVert u \rVert_{L^2(0,1)}^2 \right)^{\frac{1}{2}}. 
\end{align}
Consider the linear subspace $H^2_\sigma(0,1) \subset V_\sigma^2(0, 1)$ defined as follows:
\begin{align} \label{eq:9}
H^2_\sigma(0,1) := \left\{u \in V_\sigma^2(0, 1): \; u(0) = 0\right\}.
\end{align}

\begin{remark}\label{2r1}
Let $u \in H_\sigma^2(0,1)$. Since $u(0) = 0$, the following estimates hold:
\begin{gather}
\lVert u \rVert_{L^2(0,1)}^2 \leqslant \lVert u' \rVert_{L^2(0,1)}^2 \leqslant \frac{1}{q_0}\lVert \sqrt{q} u' \rVert_{L^2(0,1)}^2,  \label{eq:11} \\
|u'(1)|^2 \leqslant 2 \left( \frac{1}{q_0}\lVert \sqrt{q} u' \rVert_{L^2(0,1)}^2  + \dfrac{\lVert \sqrt{\sigma} u'' \rVert_{L^2(0,1)}^2}{\sigma(1)(2-K_\sigma)} \right). \label{eq:12}
\end{gather}
\end{remark}
Next, we introduce 
\begin{align} \label{eq:13}
Q_\sigma(0,1) := \left\{u \in H_\sigma^2(0,1): \, \sigma u'' \in H^2(0,1) \right\}.
\end{align}
Notice that if $u \in Q_\sigma(0,1)$, then $\sigma u''$ is continuously differentiable on $|0, \, 1]$. The Gauss-Green formula becomes:
\begin{gather} \label{eq:14}
	\int_0^1 (\sigma u'')'' v \, dx = -[ \sigma u'' v' ]_{x=0}^{x=1} + \int_{0}^{1} \sigma u'' v'' \, dx \quad \forall (u, v) \in Q_\sigma(0,1) \times H_{\sigma}^2(0,1).
\end{gather}
In order to complete the required functional framework for the qualitative analysis of system \eqref{eq:1}, we incorporate the boundary conditions $\eqref{eq:1}_2$ of the solution space, and thus define the following spaces:
\begin{align}
	H_{\sigma,0}^2(0,1) = \left\{
	\begin{array}{>{\displaystyle}l}
	 \left\{ u \in H^2_\sigma(0,1): \, u'(0) = 0  \right\} \; \text{ if } \underline{\sigma \text{ is (WD)}},\\
	 H^2_\sigma(0,1) \; \text{ if } \underline{\sigma \text{ is (SD)}};
	\end{array}
	\right. \label{eq:15} \\
	Q_{\sigma,0}(0, 1) = \left\{
	\begin{array}{>{\displaystyle}l}
	 \left\{ u \in Q_\sigma(0,1): \, u'(0) = 0  \right\}\; \text{ if } \underline{\sigma \text{ is (WD)}},\\
	 \left\{ u \in Q_\sigma(0,1): \, (\sigma u'')(0) = 0 \right\} \; \text{ if } \underline{\sigma \text{ is (SD)}}.
	\end{array}
	\right. \label{eq:16}
\end{align}

\section{Well-posedness of the closed-loop}

	\subsection{Semigroup setup}

		
Setting:
\begin{align} \label{eq:17}
w(s,t) := u_t(1, t- s \tau), \quad (s, t) \in (0, 1) \times (0, +\infty),
\end{align}
 the function $w$ satisfies the following system:
\begin{equation} \label{eq:18}
	\left\{
	\begin{array}{>{\displaystyle}l}
	\tau w_t + w_s = 0, \quad (s, t) \in (0, 1) \times (0, +\infty),\\
	w(0, t) = u_t(1, t), \quad t \in (0, \infty), \\
	w(1, t) = u_t(1, t-\tau), \quad t \in (0, \infty), \\
	w(s, 0) = g_0(-s \tau), \quad s \in (0, 1).
	\end{array}
	\right.
\end{equation}
The delayed system \eqref{eq:1} can be reformulated as follows: 
\begin{equation} \label{eq:19}
	\left\{
	\begin{array}{>{\displaystyle}l}
	 u_{tt} + \left( \sigma(x) u_{xx} \right)_{xx} - \left( q(x) u_x \right)_x = 0 , \quad (x, t) \in (0, 1) \times (0, +\infty),, \\
	\tau w_t + w_s = 0, \quad (s, t) \in (0, 1) \times (0, +\infty), \\
	u(0,t) = \mathcal{B} u(0,t) = 0, \quad t \in (0, +\infty),  \\
	- \sigma(1) u_{xx}(1,t) = \kappa_r u_x(1,t) + \kappa_a u_{xt}(1,t), \quad t \in (0, +\infty), \\
	\left( \sigma (x) u_{xx} \right)_x(1,t) - q(1) u_x(1,t) = \kappa_v u_t(1,t) + \kappa_d  w(1,t) + \kappa_b u(1,t), \quad t \in (0, +\infty), \\
	w(0, t) = u_t(1, t), \quad t \in (0, +\infty), \\
	w(1, t) = u_t(1, t-\tau), \quad t \in (0, +\infty), \\
	u(x,0) = u_0(x), \, u_t(x,0) = u_1(x), \quad x \in [0, 1], \\
	w(s, 0) = g_0(-s \tau), \quad s \in (0, 1).
	\end{array}
	\right.
\end{equation}
Given the equivalence between problems \eqref{eq:1} and \eqref{eq:19}, we will transform the latter into an abstract evolution problem in appropriate Hilbert state space. 
Consider the Hilbert space $\Big(\mathscr{H}; \, \prec \cdot, \cdot \succ\Big)$ defined by:
\begin{align}
\mathscr{H} & := H^2_{\sigma,0}(0,1) \times L^2(0,1) \times L^2(0,1), \label{eq:20}\\
\begin{split} \label{eq:21}
\prec U_1, U_2 \succ & := \int_0^1 v_1(x) v_2(x) dx + \int_0^1 q(x) u_{1}'(x) u_{2}'(x) dx \\
& \quad + \int_0^1 \sigma(x) u_{1}''(x) u_{2}''(x) dx + \, \gamma \tau \int_0^1 w_1(s) w_2(s) ds \\
& \quad + \kappa_b u_1(1) u_2(1) + \kappa_r u_{1}'(1) u_{2}'(1),
\end{split}
\end{align}
for all $U_i := (u_i, v_i, w_i) \in \mathscr{H}, \, i = 1, \, 2$, where $\gamma$ is a positive constant that we will specify later.
Next, we define a linear operator $\mathbb{A}$ on the energy space $\mathscr{H}$, whose action and domain $\mathscr{D}(\mathbb{A})$ are given by:
\begin{gather}
\mathscr{D}(\mathbb{A}) := 
\left\{
\begin{array}{l}
U = (u, v, w) \in Q_{\sigma,0}(0,1) \times H^2_{\sigma,0}(0,1) \times H^1(0,1) : \\
w(0) = v(1), \, -\sigma(1) u''(1) = \kappa_r u'(1) + \kappa_a v'(1), \\
(\sigma u'')'(1) - q(1)u'(1) = \kappa_v v(1) + \kappa_d w(1) + \kappa_b u(1)
\end{array}
\right\}, \label{eq:22} \\
\mathbb{A} U
:= \begin{pmatrix}v \\ (q u')' - (\sigma u'')'' \\ - \tau^{-1} w' \end{pmatrix}. \label{eq:23}
\end{gather}
Clearly $\mathscr{D}(\mathbb{A})$ is dense in $\mathscr{H}$.
Based on the preceding considerations, and setting $U(t) := (u(\cdot, t), \, v(\cdot, t), \, w(\cdot, t)) \in \mathscr{D}(\mathbb{A})$, we rewrite the system \eqref{eq:19} in the abstract form:
\begin{equation} \label{eq:24}
	\left\{
	\begin{array}{>{\displaystyle}l}
	 \dot{U}(t) = \mathbb{A} U(t), \; t \in (0, \infty), \\
	 U(0) = U_0 = \left(u_0, \, u_1, \, g_0(-s\tau)\right), \quad s \in (0, 1).
	\end{array}
	\right.
\end{equation}

	\subsection{Existence and uniqueness of the solution}

	This section is devoted to establishing the existence and uniqueness of the solution to the Cauchy problem \eqref{eq:24}. To this end, we show that the operator $\mathbb{A}$ is $m-$dissipative, which garantees, via the L\"umer-Phillips theorem, that it generates a $C_0-$semigroup of contractions (see \cite{pazy1983, engel2000}). \newline 
Throughout this section, we suppose \eqref{eq:4} and that:
\begin{align} \label{eq:24a}
\left|\kappa_d\right| < \gamma < 2 \kappa_v - \left|\kappa_d\right|.
\end{align}
\begin{proposition} \label{p1}
Assume that the function $\sigma$ is (WD) or (SD). The linear operator $\mathbb{A}$ defined by \eqref{eq:22}-\eqref{eq:23} is $m-$dissipative. 
\end{proposition}

\begin{proof}~\newline
\underline{$\mathbb{A}$ \itshape is dissipative}. Let $U = (u, v, w) \in \mathscr{D}(\mathbb{A})$. We have:
\begin{align} \label{eq:25}
\begin{split}
\prec \mathbb{A} U, U\succ & = \int_0^1 (qu')'(x) v(x) dx - \int_0^1 (\sigma u'')''(x) v(x) dx
 - \gamma \int_0^1 w(s) w'(s) ds \\
& \quad + \int_0^1 \sigma(x)v''(x)u''(x) dx + \int_0^1 q(x)u'(x)v'(x) dx + \kappa_b u(1) v(1) \\
& \quad  + \kappa_r u'(1) v'(1).
\end{split}
\end{align}
Integrating by parts the first three integral terms of the right hand side of the equality \eqref{eq:25} and using the boundary conditions $\eqref{eq:19}_3-\eqref{eq:19}_7$ yields:
\begin{align}  \label{eq:26}
\prec \mathbb{A} U, U\succ & = -\left(\kappa_v - \frac{\gamma}{2} \right) w^2(0) - \frac{\gamma}{2} w^2(1) - \kappa_d w(0) w(1) - \kappa_a \left(v'(1)\right)^2.
\end{align}
Applying Young's inequality to the term $- \kappa_d w(0) w(1)$ and using \eqref{eq:24a}, we deduce that:
\begin{gather} \label{eq:27}
\prec \mathbb{A} U, U\succ \leqslant -\left(\kappa_v - \frac{\gamma + \left|\kappa_d\right|}{2}\right) w^2(0) - \frac{\gamma - \left|\kappa_d\right|}{2} w^2(1) - \kappa_a \left(v'(1)\right)^2 \leqslant 0.
\end{gather}

~\newline
\underline{$\mathbb{A}$ \itshape is maximal}. 
Equivalently, the demonstration reduces to establishing the surjectivity of the linear operator $I-\mathbb{A}$. Consequently, we must prove that for all $F = (f, g, h) \in \mathscr{H}$, the equation $\left(I-\mathbb{A}\right) U = F$ admits at least one solution $U \in \mathscr{D}(\mathbb{A})$. To this end, we aim to solve the system:
\begin{gather} \label{eq:28}
\left\{
\begin{array}{l}
u - v = f \\
v - (qu')' + (\sigma u'')'' = g \\
w + \tau^{-1} w' = h.
\end{array}
\right.
\end{gather}
The equations $\eqref{eq:28}_1, \, \eqref{eq:28}_3$ are equivalent to:
\begin{gather} \label{eq:29}
\left\{
\begin{array}{>\displaystyle l}
v = u - f \\
w(s) = \left(u(1) - f(1)\right)e^{-\tau s} + \tau \int_0^s e^{(r-s)\tau} h(r) dr.
\end{array}
\right.
\end{gather}
In view of \eqref{eq:29}, the solution to \eqref{eq:28} is entirely determined by the knowledge of $u$. Consequently, $u$ satisfies the following boundary value problem:
\begin{gather} \label{eq:30}
\left\{
\begin{array}{l}
(\sigma u'')'' - (qu')' + u = f + g \\
-(\sigma u'')'(1) + q(1) u'(1) + \Lambda_1 u(1) = \Lambda_3(f, h) \\
\sigma(1) u''(1) + \Lambda_2 u'(1) = \Lambda_4(f),
\end{array}
\right.
\end{gather}
where the constants $\Lambda_1, \, \Lambda_2, \, \Lambda_3(f,h)$ and $\Lambda_4(f)$ are defined as follows:
\begin{gather} \label{eq:31}
\begin{split}
\Lambda_1 := \kappa_v + \kappa_b + \kappa_d e^{-r}, \; \Lambda_2:= \kappa_r + \kappa_a, \; \Lambda_4(f):= \kappa_a f'(1), \\
\Lambda_3(f,h) := \Big(\kappa_v + \kappa_d e^{-r}\Big)f(1) - \kappa_d \tau \int_0^1 e^{(r-1)\tau} h(r) \, dr.
\end{split}
\end{gather}
We adopt a variational approach to solve the system \eqref{eq:30}. \newline
Let $\varphi \in H^2_{\sigma,0}(0,1)$ be a test function. Multiplying equation $\eqref{eq:30}_1$ by $\varphi$, integrating by parts over $[0, \, 1]$, and subsequently incorporating the boundary conditions yields:
\begin{align} \label{eq:32}
\mathbb{B}_1(u, \varphi) = \mathbb{L}_1(\varphi), \quad \forall \varphi \in H^2_{\sigma,0}(0,1),
\end{align}
where the bilinear form $\mathbb{B}_1$ and the linear form $\mathbb{L}_1$ are defined by:
\begin{align} \label{eq:33}
\mathbb{L}_1(\varphi) & := \int_0^1 \Big(f(x) + g(x)\Big) \varphi(x) dx + \Lambda_3(f,h) \varphi(1) + \Lambda_4(f) \varphi'(1) \\
\begin{split} \label{eq:34}
\mathbb{B}_1(u, \varphi) & := \int_0^1 \Big(\sigma(x) u''(x) \varphi''(x) + q(x)u'(x)\varphi'(x) + u(x)\varphi(x)\Big)  dx \\
& \qquad + \Lambda_1 u(1) \varphi(1) + \Lambda_2 u'(1) \varphi'(1).
\end{split}
\end{align}
Direct calculations establish that $\mathbb{B}_1$ is also continuous and coercive, and $\mathbb{L}_1$ is continuous. Consequently, the Lax-Milgram theorem guarantees the existence of a unique solution $u \in H^2_{\sigma,0}(0,1)$ to equation \eqref{eq:32}. In addition the definition $v := u - f $ implies that $v \in H_{\sigma,0}^2(0,1)$.

Conversely, we now consider the weak problem \eqref{eq:32}. By  performing integrations by parts, \eqref{eq:32} becomes:
\begin{align} \label{eq:35}
\begin{split}
	& \int_0^1 \Big( (\sigma u'')''(x) - (qu')'(x) + u(x) \Big) \varphi(x) dx + \Lambda_1 u(1) \varphi(1) + \Lambda_2 u'(1) \varphi'(1) \\
& + \Big[q(x) u'(x) \varphi(x)\Big]_0^1 + \left[\sigma(x) u''(x) \varphi'(x) - \Big( \sigma u''\Big)'(x) \varphi(x) \right]_0^1 \\
& \quad = \int_0^1 \Big(f(x) + g(x)\Big) \varphi(x) dx + \Lambda_3(f,h) \varphi(1) + \Lambda_4(f) \varphi'(1).
\end{split}
\end{align}
Selecting $\varphi \in \mathcal{D}(0,1)$, the relation \eqref{eq:35} reduces to:
\begin{gather} \label{eq:36}
\int_0^1 \Big( (\sigma u'')''(x) - (qu')'(x) + u(x) \Big) \varphi(x) dx = \int_0^1 \Big(f(x) + g(x)\Big) \varphi(x) dx.
\end{gather}
Thus, we obtain $\displaystyle (\sigma u'')'' - (qu')' + u = f+g $ almost everywhere on $(0,1)$. Since $\mathcal{D}(0,1)$ is dense in $L^2(0,1)$, it follows that $\sigma u'' \in H^2(0, 1)$ and $u$ satisfies $\eqref{eq:30}_1$. Hence, $u \in Q_{\sigma}(0,1)$.
Returning into $H_{\sigma,0}^2(0,1)$ and taking account the preceding results, the equation \eqref{eq:35} becomes:
\begin{gather} \label{eq:37}
\Big( q(1)u'(1) - (\sigma u'')'(1) + \Lambda_1 u(1) - \Lambda_3(f,h) \Big) \varphi(1)  \nonumber \\
+ \Big( \sigma(1)u''(1) + \Lambda_2 u'(1) - \Lambda_4(f) \Big) \varphi'(1) - \sigma(0) u''(0) \varphi'(0) = 0.
\end{gather}
We deduce that:
\begin{gather}  \label{eq:38}
\left\{
\begin{array}{l}
\sigma(0) u''(0) \varphi'(0) = 0 \\
\Big( q(1)u'(1) - (\sigma u'')'(1) + \Lambda_1 u(1) - \Lambda_3(f,h) \Big) \varphi(1) = 0 \\
\Big( \sigma(1)u''(1) + \Lambda_2 u'(1) - \Lambda_4(f) \Big) \varphi'(1) = 0,
\end{array}
\right.
\end{gather}
for all $\varphi \in H_{\sigma,0}^2(0,1)$. Thus, the solution $u$ satisfies conditions $\eqref{eq:30}_2$ and $\eqref{eq:30}_3$ regardless of the type of degeneracy of the function $\sigma$. Moreover, if $\sigma$ is (WD), then $u'(0) = 0$ and $\eqref{eq:38}_1$ is clearly satisfied; if $\sigma$ is (SD), then $\eqref{eq:38}_1$  implies $(\sigma u'')(0) = 0$. Therefore $u \in \in Q_{\sigma, 0}(0,1)$. The operator $I-\mathbb{A}$ is surjective.
\end{proof}

We now state the first main result of this paper.
\begin{theorem}\label{th1}
Assume that the function $\sigma$ is (WD) or (SD). Under the assumptions \eqref{eq:4} and \eqref{eq:24a}, the Cauchy problem \eqref{eq:24} admits a unique solution $U \in C\left( [0, \, +\infty); \, \mathscr{H}\right)$ if the initial datum $U_0 \in \mathscr{H}$. Moreover, if $U_0 \in \mathscr{D}\left(\mathbb{A}\right)$, then $U \in C\left( [0, \, +\infty); \, \mathscr{D}\left(\mathbb{A}\right) \right) \cap C^1\left([0, \, +\infty); \, \mathscr{H}\right)$.
\end{theorem}

\begin{proof}
Applying Proposition \ref{p1}, we establish the $m-$dissipativity of the linear operator $\mathbb{A}$. This, in turn, guarantees, via the L\"umer-Phillips theorem, that $\mathbb{A}$ generates a $C_0-$semigroup of contractions on $\mathscr{H}$. We then obtain the existence and uniqueness of the solution using Hille-Yosida theorem.
\end{proof}

\section{Asymptotic behaviour of the solution}

In this section, we leverage a combined approach employing the construction of a Lyapunov functional and the energy multiplier method \cite{salhi2025, siriki2025, camasta2024, camasta2025, akil2025} to demonstrate that the solution to system \eqref{eq:1} is uniformly exponentially stable. \newline
Throughout the following analysis, we assume that the conditions \eqref{eq:4} and \eqref{eq:24a} are satisfied.

	\subsection{Dissipativity of the energy}
Inspired by anterior works \cite{salhi2025, siriki2025, dadfarnia2004, ledkim2023}, the construction of the Lyapunov functional relies on the determination of the system's energy $E(t)$, which shall be non-increasing over time.	\newline
For a regular solution $u$, we define the energy of system \eqref{eq:1} is given by:
\begin{align} \label{eq:39}
\begin{split}
E(t)  & := \frac{1}{2} \Bigg[ \int_0^1 \Big( u_t^2(x,t) + \sigma(x) u_{xx}^2(x,t) + q(x) u_x^2(x,t) \Big) \, dx + \gamma \tau \int_0^1 w^2(s,t) \, ds \\
& \qquad \quad +  \kappa_b u^2(1,t) + \kappa_r u_x^2(1,t)\Bigg],
\end{split}
\end{align}
for all $t \geqslant 0$, where $\gamma$ satisfies the relation \eqref{eq:24a}.
\begin{proposition} \label{p2}
	Let $u$ be a regular solution of problem \eqref{eq:1}, where (WS) or (SD) holds. The energy $E(t)$, defined in \eqref{eq:39}, is dissipative and satisfies
the following estimates:
\begin{align} \label{eq:40}
	\dfrac{d}{dt}E(t) & \leqslant - C_{\kappa_v, \kappa_d, \kappa_a}^{\gamma} \Big(u_t^2(1, t) + u_t^2(1, t-\tau) + u_{xt}^2(1,t) \Big),
\end{align}
where the constants $C_{\kappa_v, \kappa_d, \kappa_a}^{\gamma}$ are non-negative and given by:
\begin{align} \label{eq:41}
C_{\kappa_v, \kappa_d, \kappa_a}^{\gamma} & := \min \Big\{\kappa_a; \, \dfrac{\gamma - \left|\kappa_d \right|}{2}; \, \kappa_v - \dfrac{\gamma + \left|\kappa_d \right|}{2} \Big\}.
\end{align}
\end{proposition}

\begin{proof}
First, by multiplying equation $\eqref{eq:19}_1$ by $u_t$, and integrating it over $(0,1)$, we obtain:
\begin{gather} \label{eq:42}
\int_{0}^{1}  \Big( u_{tt} + \left( \sigma(x) u_{xx} \right)_{xx} - \left( q(x) u_x \right)_x  \Big)  u_t\, dx = 0.
\end{gather}
After applying integration by parts to the last two integral terms on the left-hand side of the preceding identity, it follows that:
\begin{align} \label{eq:43}
\begin{split}
&\frac{d}{dt} \left[ \frac{1}{2} \int_{0}^{1} \Big( u_t^2 + \sigma(x) u_{xx}^2 + q(x) u_{x}^2 \Big)\, dx \right] + \Big( (\sigma u_{xx})_x - q(1) u_x \Big)(1,t) u_t(1,t) \\
& \quad - \sigma(1) u_{xx}(1,t) u_{xt}(1,t) = 0.
\end{split}
\end{align}
Second, by multiplying equation $\eqref{eq:19}_2$ by $w$, and integrating it over $(0,1)$, we obtain:
\begin{gather} \label{eq:44}
\int_{0}^{1} w_s(s,t) w(s,t) \, ds + \tau \int_{0}^{1} w_t(s,t) w(s,t) \, ds = 0 .
\end{gather}
Integrating by parts the first integral term on the left-hand side of the preceding identity, we get:
\begin{gather} \label{eq:45}
\frac{d}{dt}\left( \frac{\tau}{2} \int_{0}^{1}  w^2(s, t) \, ds \right) + \frac{1}{2} \Big( w^2(1,t) - w^2(0,t) \Big) = 0.
\end{gather}
Summing the identities \eqref{eq:43} and \eqref{eq:45}, and incorporating the boundary conditions $\eqref{eq:19}_4-\eqref{eq:19}_7$, the following holds:
\begin{align} \label{eq:46}
\begin{split}
& \frac{d}{dt} \left[ \frac{1}{2} \left( \int_{0}^{1} \Big( u_t^2 + \sigma(x) u_{xx}^2 + q(x) u_{x}^2 \Big)\, dx + \gamma \tau \int_{0}^{1}  w^2 \, ds + \kappa_r u_x^2(1,t) + \kappa_b u^2(1,t) \right) \right]   \\
& \qquad = -\kappa_d  u_t(1,t-\tau) u_t(1,t) - \kappa_v u_t^2(1,t) - \kappa_a u_{xt}^2(1,t) - \frac{\gamma}{2} \Big( u_t^2(1,t-\tau) - u_t^2(1,t) \Big).
\end{split}
\end{align}
Furthermore, applying Young's inequality and using the condition \eqref{eq:24a} yield:
\begin{gather}
\frac{d}{dt} E(t) \leqslant -\left(\kappa_v - \frac{\gamma + \left|\kappa_d\right|}{2}\right) u_t^2(1,t) - \frac{\gamma - \left|\kappa_d\right|}{2} u_t^2(1,t-\tau) - \kappa_a \left(v'(1)\right)^2 \leqslant 0.
\end{gather}
Finally, we obtain \eqref{eq:40}.
\end{proof}

	\subsection{Exponential stability of the delayed system}

Consider the following Lyapunov function:
\begin{gather} \label{eq:47}
L(t) := E(t) + \varepsilon G(t), \quad t \geq 0,
\end{gather}
where $\varepsilon > 0$ is a parameter, whose value shall be specified to be sufficiently small later subsequently.
In addition, the functional $G$ is given by:
\begin{gather}
G(t) := \int_0^1 u_t(x,t) \Big( 2x u_x(x,t) + \frac{\Upsilon_{\sigma, \, q}}{2} u(x,t) \Big) dx + \gamma \tau \int_0^1 e^{-2\tau s} u_t^2(1, t-\tau s) ds
\end{gather}	\label{eq:48}
where $\gamma$ satisfying \eqref{eq:24a} and the positive constant $\Upsilon_{\sigma, q}$ is defined by:
\begin{gather} \label{eq:49}
\Upsilon_{\sigma, q} := \max\left\{ K_\sigma, \frac{q_2}{q_0}\right\} < 2.
\end{gather}
First, the following result demonstrates the equivalence of the norms induced by the Lyapunov functional $L$ and the energy functional $E$.
\begin{proposition} \label{p3}
Suppose that the function $\sigma$ satisfies either (WD) or (SD). For $\varepsilon > 0$ small enough, there are two positive constants $\Theta_1, \, \Theta_2$ such that:
\begin{gather} \label{eq:50}
\Theta_1 E(t) \leqslant L(t) \leqslant \Theta_2 E(t),
\end{gather}
where $\displaystyle \Theta_1 \text{ and } \, \Theta_2$ are given by:
\begin{gather} \label{eq:51}
	\Theta_1 := 1 - \varepsilon C_{\Upsilon}, \quad \Theta_2 := 1 + \varepsilon C_{\Upsilon}
\end{gather}
where the positive constant $C_{\Upsilon}$ is defined by:
\begin{gather} \label{eq:52}
C_{\Upsilon} := 2 \max \Big\{1; \, 1 + \dfrac{\Upsilon_{\sigma, q}}{4}; \, \dfrac{1}{q_0} \left(1 + \dfrac{\Upsilon_{\sigma, q}}{8} \right) \Big\}.
\end{gather}
\end{proposition}

\begin{proof}
Using Young's inequality, we obtain the following inequalities:
\begin{align} \label{eq:53}
\begin{split}
\left| \int_0^1 x u_t(x,t) u_x(x,t) dx \right| & \leqslant \frac{1}{2} \left( \lVert u_t(\cdot,t)\rVert_{L^2(0,1)}^2  + \frac{1}{q_0} \lVert \sqrt{q} u_x(\cdot,t)\rVert_{L^2(0,1)}^2  \right)
\end{split}
\end{align}
\begin{gather} \label{eq:54}
\left| \int_0^1 u_t(x,t) u(x,t) dx \right| \leqslant \frac{1}{2} \lVert u_t(\cdot,t)\rVert_{L^2(0,1)}^2 + \frac{1}{4 q_0} \lVert \sqrt{q} u_x(\cdot,t)\rVert_{L^2(0,1)}^2.
\end{gather}
Therefore, using the triangle inequality and relations \eqref{eq:53} and \eqref{eq:54} yields:
\begin{align} \label{eq:55}
\left|G(t)\right| & \leqslant \left( 1 + \frac{\Upsilon_{\sigma, q}}{4}\right)\lVert u_t(\cdot,t)\rVert_{L^2(0,1)}^2 + \frac{1}{q_0} \left(1 + \frac{\Upsilon_{\sigma, q}}{8} \right) \lVert \sqrt{q} u_x(\cdot,t)\rVert_{L^2(0,1)}^2 \nonumber \\
&	\quad + \gamma \tau \int_0^1 u_t^2(1, t-\tau s) ds \nonumber \\
\left|G(t)\right| & \leqslant C_{\Upsilon} E(t),
\end{align}
with the constant $ \displaystyle C_{\Upsilon}$ defined in \eqref{eq:52}.
Applying the triangle inequality once more to the preceding expression, we get the desired result for $0 < \varepsilon < C_{\Upsilon}^{-1}$.
\end{proof}

Before proceeding, we establish the following proposition, which is fundamental to proving the stabilization of problem \eqref{eq:1}.
\begin{proposition} \label{p4}
 Suppose that $\sigma$ is (WD) or (SD). Define:
\begin{gather} \label{eq:56}
|||y|||^2 := \int_0^1 \sigma(x) (y''(x))^2 \, dx + \int_0^1 q(x) (y'(x))^2 \, dx + \kappa_b (y(1))^2 + \kappa_r (y'(1))^2,
\end{gather}
for all $y \in H^2_{\sigma,0}(0,1)$. 
The norms $|||\cdot|||, \text{ and } \, \lVert \, \cdot  \rVert_{2, \sigma}$ 
are equivalent on $H^2_{\sigma,0}(0,1)$. In addition, for all $\lambda, \mu \in \mathbb{R}$, the variational problem:
\begin{gather} \label{eq:57}
\int_0^1 \sigma(x) y''(x)\varphi''(x) dx + \int_0^1 q(x) y'(x)\varphi'(x) dx + \kappa_b y(1)\varphi(1) + \kappa_r y'(1) \varphi'(1) = \lambda \varphi(1) + \mu \varphi'(1), 
\end{gather}
for all $\varphi \in H^2_{\sigma,0}(0,1)$, admits a unique solution $y \in H^2_{\sigma,0}(0,1)$, which satisfies the following estimates:
\begin{gather} \label{eq:58}
\lVert y \rVert^2_{L^2(0,1)} \leqslant q_0^{-1} C_{\lambda, \mu}^2 \quad \text{ and } \quad ||| y |||^2 \leqslant C_{\lambda, \mu}^2
\end{gather}
where the constant $C_{\lambda, \mu}$ is defined by:
\begin{gather} \label{eq:59}
C_{\lambda, \mu} := |\lambda| \sqrt{\dfrac{1}{q_0}} + |\mu| C_1, \quad \text{ with } \quad C_1 := \sqrt{ 2 \max \Big\{ \dfrac{1}{q_0}; \, \dfrac{1}{\sigma(1)(2 - K_\sigma)} \Big\} }.
\end{gather}
Moreover $y \in \mathcal{D}\left( A_\sigma \right) := Q_{\sigma, 0}(0, 1)$ verifies the following system:
\begin{gather} \label{eq:60}
\left\{
\begin{array}{l}
A_\sigma y = 0, \\
q(1) y'(1) - (\sigma y'')'(1) + \kappa_b y(1) = \lambda, \qquad \text{ with } A_\sigma y := (\sigma y'')'' - (q y')'.\\
\sigma(1) y''(1) + \kappa_r y'(1) = \mu,
\end{array}
\right.
\end{gather}
\end{proposition}

\begin{proof}
Let $y \in H_{\sigma,0}^2(0,1)$. First, by using inequalities \eqref{eq:11} and \eqref{eq:12}, we have:
\begin{align}
||| y |||^2 \, & \leqslant \, \left( 1 + \frac{\kappa_b + 2 \kappa_r}{q_0} \right) \left\lVert \sqrt{q} y' \right\rVert^2 + \left( 1 + \frac{2 \kappa_r}{\sigma(1)(2-K_\sigma)} \right) \left\lVert \sqrt{\sigma} y'' \right\rVert^2 \nonumber \\
||| y |||^2 \, & \leqslant \, \left( 1 + \max \left\{\frac{\kappa_b + 2 \kappa_r}{q_0}; \, \frac{2 \kappa_r}{\sigma(1)(2-K_\sigma)}\right\} \right) \left\lVert \, y \right\rVert_{2, \sigma}^2. \label{eq:60a}
\end{align}
Second, by using inequality \eqref{eq:11}, it is easy to see that: 
\[	
	\lVert y \rVert^{2}_{L^2(0,1)} \leqslant \frac{1}{2} \left\lVert y' \right\rVert^{2}_{L^2(0,1)}. 
\]
Consequently, by using \eqref{eq:12}, it follows that:
\begin{align}
\lVert y \rVert^{2}_{2, \sigma} & \leqslant \frac{3}{2} \left\lVert y' \right\rVert^2 + \left\lVert \sqrt{\sigma} y'' \right\rVert^2 \nonumber \\
& \leqslant \frac{3}{2 q_0} \left\lVert \sqrt{q} y' \right\rVert^2 + \left\lVert \sqrt{\sigma}  y'' \right\rVert^2 \nonumber \\
\lVert y \rVert^{2}_{2, \sigma} & \leqslant \max \left\{1; \, \frac{3}{2q_0}\right\} ||| y |||^2. \label{eq:60b}
\end{align}
Thus, the estimates \eqref{eq:60a} and \eqref{eq:60b} ensure that the norms $\lVert \cdot \rVert_{2, \sigma}$ and $||| \cdot |||$ are equivalent on $H^2_{\sigma,0}(0,1)$.

Let $\varphi \in H_{\sigma,0}^2(0,1)$ be a test function. Multiplying the equation $\eqref{eq:60}_1$ by $\varphi$ and performing integration by parts twice while incorporating the boundary conditions $\eqref{eq:60}_2$ and $\eqref{eq:60}_3$, we immediately derive the weak formulation:
\begin{gather} \label{eq:61}
\chi (y, \varphi) = \Psi(\varphi), \quad \forall \varphi \in H_{\sigma,0}^2(0,1),
\end{gather}
where the applications $\Psi$ and $\chi$ are defined as follows:
\begin{gather} \label{eq:62}
\Psi(\varphi) := \lambda \varphi(1) + \mu \varphi'(1), \\
\chi (y, \varphi) := \int_{0}^{1} \sigma y'' \varphi'' dx + \int_{0}^{1} q y' \varphi' \, dx + \kappa_b y(1) \varphi(1)  + \kappa_r y'(1) \varphi'(1). \label{eq:63}
\end{gather}
A direct calculation shows that the bilinear form $\chi$ is continuous and coercive, and that the linear form $\Psi$ is also continuous. Applying the Lax-Milgram Theorem, the variational problem \eqref{eq:61} admits a unique solution $y \in H_{\sigma,0}^2(0,1)$; in particular, we have:
\begin{align} \label{eq:63a}
\begin{split}
|||y|||^2 & = \int_0^1 \Big( \sigma(x)(y''(x))^2 + q(x) (y'(x))^2 \Big) \, dx + \kappa_b (y(1))^2 + \kappa_r (y'(1))^2 \\
& = \lambda y(1) + \mu y'(1).
\end{split}
\end{align}
 Using the inequalities \eqref{eq:11} and \eqref{eq:12}, it follows that:
\begin{align} \label{eq:64}
|||y|||^2 & \leqslant C_{\lambda, \mu} \, |||y|||.
\end{align}
where $\displaystyle C_{\lambda, \mu}$ is defined in \eqref{eq:59}.
In addition, we have:
\begin{align*}
\lVert y \rVert_{L^2(0,1)}^2 \leqslant \frac{1}{q_0} |||y|||^2 \leqslant \frac{1}{q_0} C_{\lambda, \mu}^2.
\end{align*}
Then \eqref{eq:58} holds.

Conversely, considering $y$ as the solution to the weak problem \eqref{eq:61}, and performing two successive integrations by parts, the resulting expression simplifies to:
\begin{align} \label{eq:65}
\begin{split}
& \int_{0}^{1} \Big( (\sigma y'')'' - (qy')' \Big) \varphi \, dx + \Big( -(\sigma y'')'(1) + q(1) y'(1) + \kappa_b y(1) \Big) \varphi(1)  \\
& + \Big( \kappa_r y'(1) + \sigma(1) y''(1) \Big) \varphi'(1) - \sigma(0) y''(0) \varphi'(0) = \lambda \varphi(1) + \mu \varphi'(1).
\end{split}
\end{align}
Selecting $\varphi \in \mathcal{D}(0,1)$, the expression \eqref{eq:65} becomes:
\begin{align} \label{eq:66}
\begin{split}
\int_{0}^{1} \Big( (\sigma(x) y'')'' - (q(x)y')' \Big) \varphi \, dx = 0.
\end{split}
\end{align}
Consequently, $\displaystyle (\sigma y'')'' - (q y')' = 0$ almost everywhere (a.e.) in $(0,1)$. By the density of $\mathcal{D}(0,1)$ in $L^2(0,1)$, the identity $\eqref{eq:60}_1$ holds. This establishes that $\sigma y'' \in H^2(0,1)$, which yields the regularity $y \in Q_{\sigma}(0,1)$.\newline
On the other hand, considering $\varphi \in H_{\sigma,0}^2(0,1)$, it follows from the preceding that:
\begin{align} \label{eq:67}
\begin{split}
& \Big( -(\sigma y'')'(1) + q(1) y'(1) + \kappa_b y(1) \Big) \varphi(1)   + \Big( \kappa_r y'(1) + \sigma(1) y''(1) \Big) \varphi'(1) \\
& \quad - \sigma(0) y''(0) \varphi'(0) = \lambda \varphi(1) + \mu \varphi'(1).
\end{split}
\end{align}
By identification, we obtain:
\begin{gather} \label{eq:68}
\left\{
\begin{array}{l}
\sigma(0) y''(0) \varphi'(0) = 0, \\
\Big( \kappa_r y'(1) + \sigma(1) y''(1) \Big) \varphi'(1) = \mu \varphi'(1), \\
\Big( -(\sigma y'')'(1) + q(1) y'(1) + \kappa_b y(1) \Big) \varphi(1) = \lambda \varphi(1),
\end{array}
\right.
\end{gather}
for all $\varphi \in H^2_{\sigma, 0}(0,1)$. Since $\varphi$ is an arbitrary test function in $H^2_{\sigma, 0}(0,1)$, the boundary conditions $\eqref{eq:60}_1$ and $\eqref{eq:60}_2$ are satisfied. 
Furthermore, if $\sigma$ is (WD), then $y'(0) = 0$ and  the relation $\eqref{eq:68}_1$ is satisfied. Otherwise, if $\sigma$ is (SD), the idendity simplifies to $\sigma(0) y''(0) = 0$. Therefore $y \in Q_{\sigma, 0}(0,1)$ and solves the boundary value problem \eqref{eq:60}.
\end{proof}

To state the second main result of this paper, we need the following lemmas.\begin{lemma} \label{l1}
Let $\sigma$ satisfy either (WD) or (SD). Then, for any regular solution $u$ of system \eqref{eq:1}, we have the following:
\begin{align} \label{eq:69}
\dfrac{d}{dt} G(t) & \leqslant -\min\Big\{ 2 - \Upsilon_{\sigma, q}; \, 4e^{-2\tau} \Big\} E(t) + \left[ \left( \dfrac{2}{\sigma(1)} + \dfrac{1}{q(1)} \left( 2+\dfrac{\Upsilon_{\sigma, q}}{2} \right)\right) \kappa_r^2  + \left( \dfrac{3}{2}+\dfrac{\Upsilon_{\sigma, q}}{4} \right)q(1) \right] u_x^2(1,t)  \nonumber \\
& \quad  + \left( \dfrac{2}{\sigma(1)} + \dfrac{1}{q(1)} \left( 2+\dfrac{\Upsilon_{\sigma, q}}{2} \right) \right) \kappa_a^2 u_{xt}^2(1,t) + \left( \dfrac{3}{q(1)}\kappa_d^2-\gamma e^{-2\tau} \right)  u_t^2(1,t-\tau) \nonumber \\
& \quad  + \left( 1 + \gamma + \dfrac{3}{q(1)} \kappa_v^2 \right) u_t^2(1,t) + \left( \dfrac{3}{q(1)} \kappa_b^2 + \dfrac{q(1)}{4} \Upsilon_{\sigma, q}^2 \right) u^2(1,t).
\end{align}
\end{lemma}

\begin{proof}
On the one hand, we have:
\begin{gather} \label{eq:70}
\frac{d}{dt}\left( \int_0^1 x u_t(x,t)  u_x(x,t) dx \right) = \int_0^1 x \Big( u_{tt}(x,t) u_x(x,t) + u_t(x,t) u_{xt}(x,t)\Big) dx.
\end{gather}
Using the equation $\eqref{eq:1}_1$, we have:
\begin{gather} \label{eq:71}
\int_0^1 x u_{tt}(x,t) u_x(x,t) \, dx =  \int_0^1 x (q u_x)_x(x,t) u_x(x,t) \, dx - \int_0^1 x (\sigma u_{xx})_{xx}(x, t) u_x(x,t) \, dx.
\end{gather}
Next, by proceeding by integrations by parts, we obtain the following identities:
\begin{align} \label{eq:72}
\int_0^1 x (q u_x)_x(x,t) u_x(x,t) dx = \frac{1}{2} q(1) u_x^2(1,t) +\frac{1}{2} \int_0^1 \Big( -q(x) + xq'(x) \Big) u_x^2(x,t) dx
\end{align}
%
\begin{align} \label{eq:73}
\begin{split}
& \int_0^1 x (\sigma u_{xx})_{xx}(x, t) u_x(x,t) dx  = \frac{1}{2}\int_0^1 \Big( 3 \sigma(x) - x \sigma'(x) \Big) u_{xx}^2(x,t) dx  \\
& \qquad +  (\sigma u_{xx})_x(1,t) u_x(1,t) - \frac{1}{2} \sigma(1) u_{xx}^2(1,t) - \sigma(1) u_{x}(1,t) u_{xx}(1,t).
\end{split}
\end{align}
Thus, utilizing expressions \eqref{eq:72} and \eqref{eq:73}, we deduce that:
\begin{align} \label{eq:74}
\begin{split}
& \int_0^1 x u_{tt}(x,t) u_x(x,t) dx \\
& \quad = \frac{1}{2}\int_0^1 \left( -3 \sigma(x) + x \sigma'(x) \right) u_{xx}^2(x,t) dx + \frac{1}{2} \int_0^1 \Big( -q(x) + xq'(x) \Big) u_x^2(x,t) dx \\
& \qquad - (\sigma u_{xx})_x(1,t) u_x(1,t)  + \sigma(1) u_{x}(1,t) u_{xx}(1,t) + \frac{1}{2} \sigma(1) u_{xx}^2(1,t) + \frac{1}{2} q(1) u_x^2(1,t).
\end{split}
\end{align}
In addition, integration by parts and incorporation of the boundary condition $\eqref{eq:1}_2$ yields:
\begin{align} \label{eq:75}
\int_0^1 x u_t(x,t) u_{xt}(x,t) dx & = \frac{1}{2} u_t^2(1,t) - \frac{1}{2} \int_0^1 u_t^2(x,t) dx.
\end{align}
Therefore, the sum of identities \eqref{eq:74} and \eqref{eq:75} yields the expression of the time derivative of the first integral term of $G(t)$:
\begin{align} \label{eq:76}
\begin{split}
& \frac{d}{dt}\left( \int_0^1 x u_t(x,t)  u_x(x,t) dx \right)  \\
& \quad = - \frac{1}{2} \int_0^1 u_t^2(x,t) dx + \frac{1}{2}\int_0^1 \Big( -3 \sigma(x) + x \sigma'(x) \Big) u_{xx}^2(x,t) dx \\
& \qquad + \frac{1}{2} \int_0^1 \Big( -q(x) + xq'(x) \Big) u_x^2(x,t) dx + \frac{1}{2} u_t^2(1,t) + \frac{1}{2} q(1) u_x^2(1,t)  \\
& \qquad + \frac{1}{2} \sigma(1) u_{xx}^2(1,t) + \Big( \sigma(1) u_{xx}(1,t)-  (\sigma u_{xx})_x(1,t) \Big) u_x(1,t).
\end{split}
\end{align}
On the other hand, we have:
\begin{align} \label{eq:77}
\frac{d}{dt} \left( \int_0^1 u_t(x,t) u(x,t) dx \right) & = \int_0^1 \Big( (qu_x)_x - (\sigma u_{xx})_{xx} \Big)(x,t) u(x,t) dx + \int_0^1  u_t^2(x,t) dx.
\end{align}
Applying integrations by parts yields:
\begin{align} \label{eq:78}
	 \int_0^1 (qu_x)_x(x,t) u(x,t) dx & = q(1) u_x(1,t) u(1,t) - \int_0^1 q(x) u_x^2(x,t) dx \\
\begin{split} \label{eq:79}
\int_0^1 (\sigma u_{xx})_{xx}(x,t) u(x,t) dx & =  \int_0^1 \sigma(x) u_{xx}^2(x,t) dx + (\sigma u_{xx})_x(1,t) u(1,t) \\
	& \qquad -\sigma(1) u_{xx}(1,t) u_x(1,t).
\end{split}
\end{align}
Using the identities \eqref{eq:78} and \eqref{eq:79}, the relation \eqref{eq:77} becomes:
\begin{align} \label{eq:80}
\begin{split}
& \frac{d}{dt} \left( \int_0^1 u_t(x,t) u(x,t) dx \right) \\
& \quad = \int_0^1 u_t^2(x,t) dx - \int_0^1 q(x) u_x^2(x,t) - \int_0^1 \sigma(x) u_{xx}^2(x,t) dx  \\
& \qquad + \Big( q(1) u_x(1,t) - (\sigma u_{xx})_x(1,t) \Big) u(1,t) + \sigma(1) u_{xx}(1,t) u_x(1,t).
\end{split}
\end{align}
Finally, combining \eqref{eq:76} and \eqref{eq:80} yields:
\begin{align} \label{eq:81}
& \frac{d}{dt} \left( \int_0^1 u_t(x,t) \Big( 2x u_x(x,t) + \frac{\Upsilon_{\sigma, \, q}}{2} u(x,t) \Big) dx \right) \nonumber\\
& = \int_0^1 \left( -\Big( 1+\frac{\Upsilon_{\sigma, \, q}}{2} \Big) q(x) + x q'(x) \right) u_x^2(x,t) \, dx + \int_0^1 \Big( -\left(3 + \frac{\Upsilon_{\sigma, q}}{2} \right) \sigma(x) + x \sigma'(x) \Big) u_{xx}^2(x,t) \, dx \nonumber \\
& \quad + \left( -1 + \frac{\Upsilon_{\sigma, \, q}}{2} \right) \int_0^1 u_t^2(x,t) \, dx + q(1) u_x^2(1,t) + u_t^2(1,t) + \sigma(1) u_{xx}^2(1,t) - 2 (\sigma u_{xx})_x(1,t) u_x(1,t) \nonumber \\
& \quad + \left(2 + \frac{\Upsilon_{\sigma, \, q}}{2}\right) \sigma(1) u_{xx}(1,t) u_x(1,t) + \frac{\Upsilon_{\sigma, \, q}}{2} \Big( q(1) u_x(1,t) - (\sigma u_{xx})_x(1,t) \Big) u(1,t).
\end{align}
Owing to the definition of $\Upsilon_{\sigma, \, q}$ in \eqref{eq:49}, it follows that:
\begin{gather} \label{eq:82}
\begin{split}
-\Big( 1+\frac{\Upsilon_{\sigma, \, q}}{2} \Big) q(x) + x q'(x) \leqslant 0 \\
- \left( 3+\frac{\Upsilon_{\sigma, q}}{2} \right) \sigma(x) + x \sigma'(x) \leqslant \left( -1 + \frac{\Upsilon_{\sigma, \, q}}{2} \right) \sigma(x) \leqslant 0.
\end{split}
\end{gather}
Consequently, we obtain the following estimate:
\begin{align} \label{eq:83}
& \frac{d}{dt} \left( \int_0^1 u_t(x,t) \Big( 2x u_x(x,t) + \frac{\Upsilon_{\sigma, \, q}}{2} u(x,t) \Big) dx \right) \nonumber \\
& \leqslant \Big( -1 + \frac{\Upsilon_{\sigma, \, q}}{2} \Big) \left( \int_0^1 u_t^2(x,t) \, dx + \int_0^1 q(x) u_x^2(x,t) \, dx + \int_0^1 \sigma(x) u_{xx}^2(x,t) \, dx \right) \nonumber \\
& \quad + u_t^2(1,t) + \underbrace{\sigma(1) u_{xx}^2(1,t)}_{(i)} + \underbrace{\left(2 + \frac{\Upsilon_{\sigma, \, q}}{2}\right) \sigma(1) u_{xx}(1,t) u_x(1,t)}_{(ii)} \underbrace{- (\sigma u_{xx})_x(1,t) u_x(1,t)}_{(iii)} \nonumber \\
& \quad + \underbrace{\Big( q(1) u_x(1,t) - (\sigma u_{xx})_x(1,t) \Big) \Big( \frac{\Upsilon_{\sigma, \, q}}{2} u(1,t) + u_x(1,t) \Big)}_{(iv)}.
\end{align}
By applying the boundary conditions $\eqref{eq:1}_3$ and $\eqref{eq:1}_4$ to terms $(i), \, (ii), \, (iii)$ and $(iv)$ and, in turn, using Young's inequality yield the following inequalities:
\begin{align} \label{eq:84}
	(i) & \leqslant 2 \left(\sigma(1)\right)^{-1} \Big( \kappa_r^2 u_x^2(1,t) + \kappa_a^2 u_{xt}^2(1,t) \Big) \\ \label{eq:85}
	(ii) & \leqslant \frac{1}{2q(1)}\left(2 + \frac{\Upsilon_{\sigma, \, q}}{2}\right) \Big( \left( q^2(1) + 2 \kappa_r^2 \right) u_x^2(1,t) + 2 \kappa_a^2 u_{xt}^2(1,t)\Big) \\ \label{eq:86}
	(iii) & \leqslant -\frac{1}{2} q(1)u_x^2(1,t) + \frac{3}{2q(1)} \Big( \kappa_v^2 u_t^2(x,1) + \kappa_d^2 u_t^2(1, t-\tau) + \kappa_b^2 u^2(1,t) \Big) \\
\begin{split} \label{eq:87}
	(iv) & \leqslant \frac{3}{2q(1)} \Big( \kappa_v^2 u_t^2(x,1) + \kappa_d^2 u_t^2(1, t-\tau) \Big) + \left( q(1)\frac{\Upsilon_{\sigma, \, q}^2}{4} + \frac{3}{2q(1)} \kappa_b^2 \right) u^2(1,t) \\
& \qquad + q(1) u_x^2(1,t).
\end{split}
\end{align}
Hence, in view of the preceding expressions \eqref{eq:84}-\eqref{eq:87}, we rewrite \eqref{eq:83} as :
\begin{align} \label{eq:88}
\begin{split}
& \frac{d}{dt} \left( \int_0^1 u_t(x,t) \Big( 2x u_x(x,t) + \frac{\Upsilon_{\sigma, \, q}}{2} u(x,t) \Big) dx \right) \\
& \quad \leqslant \Big( -1 + \frac{\Upsilon_{\sigma, \, q}}{2} \Big) \left( \int_0^1 u_t^2(x,t) \, dx + \int_0^1 q(x) u_x^2(x,t) \, dx + \int_0^1 \sigma(x) u_{xx}^2(x,t) \, dx \right) \\
& \quad + \left( 1 + \frac{3}{q(1)} \kappa_v^2 \right) u_t^2(1,t) + \left[ \left( \frac{2}{\sigma(1)} + \dfrac{1}{q(1)} \left( 2 + \dfrac{\Upsilon_{\sigma, q}}{2} \right) \right)\kappa_r^2 + \left( \frac{3}{2}+\frac{\Upsilon_{\sigma, q}}{4} \right) q(1) \right] u_x^2(1,t) \\
& \quad + \left( \frac{2}{\sigma(1)} + \dfrac{1}{q(1)} \left( 2 + \dfrac{\Upsilon_{\sigma, q}}{2} \right) \right) \kappa_a^2 u_{xt}^2(1,t) + \left( \frac{3}{q(1)} \kappa_b^2 + \frac{q(1)}{4} \Upsilon_{\sigma, q}^2 \right) u^2(1,t) \\
& \quad + \frac{3}{q(1)} \kappa_d^2 u_t^2(1,t-\tau).
\end{split}
\end{align}
Moreover, the time derivative of the last integral term of the function $G(t)$ is:
\begin{align} \label{eq:89}
\begin{split}
\frac{d}{dt} \left( \int_0^1 e^{-2\tau s} u^2_t(1, t-\tau s) ds \right) =& \; \tau^{-1} \left( u_t^2(1,t) - e^{-2\tau} u_t^2(1, t-\tau) \right) \\
& - 2 \int_0^1 e^{-2\tau s} u_t^2(1,t-\tau s) ds.
\end{split}
\end{align}
We now estimate the time derivative of the function $G(t)$. Applying the triangle inequality and combining it with the time derivative estimates \eqref{eq:88} and \eqref{eq:89} of the integral terms in $G(t)$ yields:
\begin{align} \label{eq:90}
\begin{split}
	\frac{d}{dt} G(t) & \leqslant \Big( -1 + \frac{\Upsilon_{\sigma, \, q}}{2} \Big) \left( \int_0^1 u_t^2(x,t) \, dx + \int_0^1 q(x) u_x^2(x,t) \, dx + \int_0^1 \sigma(x) u_{xx}^2(x,t) \, dx \right) \\
	& \quad - 2\gamma \tau \int_0^1 e^{-2\tau s} u_t^2(1,t-\tau s) ds  + \left( \frac{3}{q(1)}\kappa_d^2-\gamma e^{-2\tau} \right)  u_t^2(1,t-\tau) \\
	& \quad + \left( 1 + \gamma + \frac{3}{q(1)} \kappa_v^2 \right) u_t^2(1,t) + \left( \frac{3}{q(1)} \kappa_b^2 + \frac{q(1)}{4} \Upsilon_{\sigma, q}^2 \right) u^2(1,t) \\
	& \quad + \left[ \left( \frac{2}{\sigma(1)} + \dfrac{1}{q(1)} \left( 2 + \dfrac{\Upsilon_{\sigma, q}}{2} \right) \right)\kappa_r^2 + \left( \frac{3}{2}+\frac{\Upsilon_{\sigma, q}}{4} \right) q(1) \right] u_x^2(1,t) \\
& \quad + \left( \frac{2}{\sigma(1)} + \dfrac{1}{q(1)} \left( 2 + \dfrac{\Upsilon_{\sigma, q}}{2} \right) \right) \kappa_a^2 u_{xt}^2(1,t).
\end{split}
\end{align}
Thereby leading us to the desired result \eqref{eq:69}.
\end{proof}


\begin{lemma} \label{l2}
Let $T>0$. Assume that the function $\sigma$ is either (WD) or (SD). For $\varepsilon > 0$ small enough, we have for any $r \in (0,T)$:
\begin{gather} \label{eq:91}
\varepsilon \min\Big\{ 2 - \Upsilon_{\sigma, q}; \, 4e^{-2\tau} \Big\} \int_r^T E(t) \, dt \leqslant L(r) - L(T) + \varepsilon C_0 \left[ \int_r^T u^2(1,t) dt + \int_r^T u_x^2(1,t) dt \right].
\end{gather}
where $C_0$ is a positive constant given by:
\begin{align} \label{eq:92}
C_0 := \max \Bigg\{\dfrac{3}{q(1)} \kappa_b^2 + \dfrac{q(1)}{4} \Upsilon_{\sigma, q}^2 , \; \left( \dfrac{2}{\sigma(1)} + \dfrac{1}{q(1)} \left( 2 + \dfrac{\Upsilon_{\sigma, q}}{2} \right) \right)\kappa_r^2 + \left( \frac{3}{2}+\frac{\Upsilon_{\sigma, q}}{4} \right) q(1) \Bigg\}.
\end{align}
\end{lemma}

\begin{proof} 
Using expressions \eqref{eq:69} and \eqref{eq:40}, we derive the following estimates:
\begin{align} \label{eq:93}
& 	\frac{d}{dt} L(t) \leqslant  - C_{\kappa_v, \kappa_d, \kappa_a}^{\gamma} \Big(u_t^2(1, t) + u_t^2(1, t-\tau) + u_{xt}^2(1,t) \Big) \nonumber \\
& \quad + \varepsilon \left[ -\min\Big\{ 2 - \Upsilon_{\sigma, q}; \, 4e^{-2\tau} \Big\} E(t) + \frac{3}{q(1)}\kappa_d^2  u_t^2(1,t-\tau) + \left( 1 + \gamma + \frac{3}{q(1)} \kappa_v^2 \right) u_t^2(1,t) \right. \nonumber\\
& \quad + \left[ \left( \frac{2}{\sigma(1)} + \dfrac{1}{q(1)} \left( 2 + \dfrac{\Upsilon_{\sigma, q}}{2} \right) \right)\kappa_r^2 + \left( \frac{3}{2}+\frac{\Upsilon_{\sigma, q}}{4} \right) q(1) \right] u_x^2(1,t) \nonumber \\
& \quad \left. + \left( \frac{2}{\sigma(1)} + \dfrac{1}{q(1)} \left( 2 + \dfrac{\Upsilon_{\sigma, q}}{2} \right) \right) \kappa_a^2 u_{xt}^2(1,t)  + \left( \frac{3}{q(1)} \kappa_b^2 + \frac{q(1)}{4} \Upsilon_{\sigma, q}^2 \right) u^2(1,t) \right] \nonumber \\
& \frac{d}{dt} L(t) \leqslant - \varepsilon \min\Big\{ 2 - \Upsilon_{\sigma, q}; \, 4e^{-2\tau} \Big\} E(t) - C_\varepsilon \Big(u_t^2(1, t) + u_t^2(1, t-\tau) + u_{xt}^2(1,t) \Big) \nonumber \\
& \qquad \qquad + \varepsilon C_0 \Big( u^2(1,t) + u_x^2(1,t) \Big) 
\end{align}
where the positive constant $C_0$ is defined in \eqref{eq:92} and the expression of $C_\varepsilon$ is given by:
\begin{align} \label{eq:96}
\begin{split}
C_\varepsilon & := \min \Bigg\{C_{\kappa_v, \kappa_d, \kappa_a}^{\gamma} - \varepsilon \left( 1 + \gamma + \frac{3}{q(1)} \kappa_v^2 \right), \, C_{\kappa_v, \kappa_d, \kappa_a}^{\gamma} - \varepsilon \frac{3}{q(1)}\kappa_d^2,  \\
 & \qquad \qquad \quad C_{\kappa_v, \kappa_d, \kappa_a}^{\gamma} - \varepsilon \left( \frac{2}{\sigma(1)} + \dfrac{1}{q(1)} \left( 2 + \dfrac{\Upsilon_{\sigma, q}}{2} \right) \right) \kappa_a^2 \Bigg\} > 0, \\
\end{split}
\end{align}
for $\varepsilon$ chosen sufficiently small such that:
\begin{gather} \label{eq:97}
0 < \varepsilon < \min\left\{ \dfrac{C_{\kappa_v, \kappa_d, \kappa_a}^\gamma}{\dfrac{3}{q(1)}\kappa_d^2}, \; \dfrac{C_{\kappa_v, \kappa_d, \kappa_a}^\gamma}{1 + \gamma + \dfrac{3}{q(1)} \kappa_v^2}, \; \dfrac{C_{\kappa_v, \kappa_d, \kappa_a}^\gamma}{+ \left( \dfrac{2}{\sigma(1)} + \dfrac{1}{q(1)} \left( 2 + \dfrac{\Upsilon_{\sigma, q}}{2} \right) \right) \kappa_a^2} \right\}.
\end{gather}
Finally, we obtain the following estimate, regardless of the choice of $\kappa_a$:
\begin{align}
\frac{d}{dt} L(t) \leqslant - \varepsilon \min\Big\{ 2 - \Upsilon_{\sigma, q}; \, 4e^{-2\tau} \Big\} E(t) + \varepsilon C_0 \Big( u^2(1,t) + u_x^2(1,t) \Big).
\end{align}
Let $r \in (0, T)$. Integrating inequality \eqref{eq:93} over $(r, T)$ with respect to time, the result \eqref{eq:91} follows directly.
\end{proof}


\begin{lemma} \label{l3}
Suppose that $\sigma$ satisfies (WD) or (SD) and consider $\kappa_b, \, \kappa_r > 0$. Let $\delta$ be a positive constant given by:
\begin{align}
\delta := \dfrac{1}{2} \left[ \left( \dfrac{1}{\kappa_b} + \frac{1}{\kappa_r} \right) \max \left\{ \frac{1}{q_0}, \, C_1^2\right\} \right]^{-1}. \label{eq:97a}
\end{align}
Then, for any regular solution $u$ of the delayed system \eqref{eq:1}, we have the following estimate:
\begin{align} \label{eq:98}
& \int_r^T \left( u^2(1,t) + u_x^2(1,t) \right) \, dt \leqslant 2 \left[ \tilde{\delta} \int_r^T E(t) dt + C_2^\delta \Big( E(r) - E(T) \Big) + C_3 \Big( E(T) + E(r) \Big) \right],
\end{align}
for every $\tilde{\delta} > 0$, where $C_2^\delta$ and $C_3$ are defined as follows:
\begin{gather} \label{eq:99}
C_2^\delta := 
\dfrac{\max \left\{ \dfrac{\kappa_v^2}{\delta}; \, \dfrac{\kappa_d^2}{\delta}; \, \dfrac{\kappa_a^2}{2\delta} \right\}}{C_{\kappa_v, \kappa_d, \kappa_a}^{\gamma}} + \dfrac{\max\left\{ \dfrac{1}{q_0^2}; \, \dfrac{ C_1^2}{q_0} \right\}}{\tilde{\delta} C_{\kappa_v, \kappa_d, \kappa_a}^{\gamma}}, \;
C_3 := \max \left\{ 1; \, \dfrac{4}{\kappa_b \, q_0^2}; \, \dfrac{4 C_1^2}{\kappa_r \, q_0} \right\}.
\end{gather}
\end{lemma}

\begin{proof}
Set $\lambda := u(1, t), \, \mu := u_x(1, t)$ and let $y := y(\cdot, t) \in H^2_{\sigma, 0}(0, 1)$ be the unique solution of the following equation:
\begin{align} \label{eq:100}
\begin{split}
& \int_0^1 \sigma(x) y_{xx} \varphi_{xx} dx + \int_0^1 q(x) y_x \varphi_x dx + \kappa_b y(1, t)\varphi(1) + \kappa_r y_x(1, t) \varphi_x(1) \\
  & \qquad = \lambda \varphi(1) + \mu \varphi_x(1), \quad \forall \varphi \in H^2_{\sigma,0}(0,1).
\end{split}
\end{align}
Due to Proposition \ref{p4}, $y \in \mathcal{D}(\mathbb{A}_\sigma)$ and solves the following system: 
\begin{gather} \label{eq:101}
\left\{
\begin{array}{l}
\left( \sigma y_{xx} \right)_{xx} - (q y_x)_x = 0, \\
q(1) y_x(1,t) - (\sigma y_{xx})_x(1,t) + \kappa_b y(1,t) = \lambda, \\
\sigma(1) y_{xx}(1,t) + \kappa_r y_x(1,t) = \mu.
\end{array}
\right. 
\end{gather}
First, multiplying $\eqref{eq:1}_1$ by $y$ and integrating it over $(r,T)\times(0,1)$ yields:
\begin{gather} \label{eq:102}
\int_r^T \int_0^1 u_{tt} y \, dx \, dt + \int_r^T \int_0^1 \Big[ \left( \sigma(x) u_{xx} \right)_{xx} - \left( q(x) u_x \right)_x \Big] y \, dx \, dt = 0.
\end{gather}
Performing integration by parts,
we obtain:
\begin{gather} 
\left[ \int_0^1 u_t y dx \right]_{t=r}^{t=T} + \int_r^T \Big( \left( \sigma u_{xx} \right)_x(1,t) - q(1) u_x(1,t) \Big) y(1,t) dt - \int_r^T \sigma(1) u_{xx}(1,t) y_x(1,t) dt \nonumber \\
- \int_r^T \int_0^1 u_{t} y_t dx \, dt  +  \int_r^T \int_0^1 \sigma(x) u_{xx} y_{xx} dx \, dt + \int_r^T \int_0^1 q(x) u_{x} y_x dx \, dt = 0. \label{eq:103}
\end{gather}
Second, multiplying the first equation of \eqref{eq:101} by $u$ and integrating it over $(r,T)\times(0,1)$ by parts, the following hold:
\begin{gather} 
\int_r^T \int_0^1 \Big( \left( \sigma y_{xx} \right)_{xx} - (q y_x)_x \Big) u \, dx \, dt = 0. \nonumber \\
\int_r^T \int_0^1 \sigma(x) y_{xx} u_{xx} \, dx \, dt + \int_r^T \int_0^1 q(x) u_x y_x \, dx \, dt  \nonumber \\ 
+ \, \int_r^T \left[ \Big( \left(\sigma y_{xx} \right)_{x}(1, t) - q(1) y_x(1, t) \Big) u(1,t) - \sigma(1) y_{xx}(1,t) u_x(1,t) \right] \, dt = 0. \label{eq:104}
\end{gather}
Using the identity \eqref{eq:104}, the relation \eqref{eq:103} becomes:
\begin{align}
\begin{split}
& \int_r^T \int_0^1 u_{t} y_t dx \, dt - \left[ \int_0^1 u_t y dx \right]_{t=r}^{t=T} \\
& \quad = \int_r^T \left[ \Big( \left( \sigma u_{xx} \right)_x(1,t) - q(1) u_x(1,t) \Big) y(1,t) -  \sigma(1) u_{xx}(1,t) y_x(1,t) \right] dt \nonumber \\
& \qquad - \int_r^T \left[ \Big( \left(\sigma y_{xx} \right)_{x}(1,t) - q(1) y_x(1,t) \Big) u(1,t) - \sigma(1) y_{xx}(1,t) u_x(1,t) \right] \, dt.
\end{split}
\end{align}
Using the boundary conditions $\eqref{eq:1}_3, \, \eqref{eq:1}_4, \, \eqref{eq:101}_2$ and $\eqref{eq:101}_3$, and substituting $\lambda$ and $\mu$ by their expressions, we obtain:
\begin{align} \label{eq:105}
\begin{split}
& \int_r^T \int_0^1 u_{t} y_t dx \, dt - \left[ \int_0^1 u_t y dx \right]_{t=r}^{t=T} = \int_r^T \Big( \kappa_v u_t(1,t) + \kappa_d  u_t(1,t-\tau) \Big) y(1,t) dt \\
& \qquad  + \int_r^T \kappa_a u_{xt}(1,t) y_x(1,t) dt +\int_r^T \left( u^2(1,t) + u_x^2(1,t) \right) \, dt.
\end{split}
\end{align}
Then:
\begin{align} \label{eq:106}
\begin{split}
& \int_r^T \left( u^2(1,t) + u_x^2(1,t) \right) \, dt =  \int_r^T \int_0^1 u_{t} y_t dx \, dt - \left[ \int_0^1 u_t y dx \right]_{t=r}^{t=T}  \\
& \qquad - \int_r^T \kappa_a \, u_{xt}(1,t) y_x(1,t) dt - \int_r^T \Big( \kappa_v \, u_t(1,t) + \kappa_d \, u_t(1,t-\tau) \Big) y(1,t) dt.
\end{split}
\end{align}
Next, we need to estimate the integral terms on the right-hand side of the preceding identity. 
Using Young's inequality, we have:
\begin{align} \label{eq:107}
\left| \int_0^1 u_t(x,t) y(x,t) \, dx \right| & \leqslant \frac{1}{2} \left( \int_0^1 u_t^2(x,t) dx + \int_0^1 y(x,t)^2 dx \right).
\end{align}
From \eqref{eq:58}, we get:
\begin{align} \label{eq:108}
\lVert y(\cdot,t) \rVert^2_{L^2(0,1)}	& \leqslant \frac{2}{q_0} \left( \frac{1}{q_0} u^2(1,t) + C_1^2 u^2_x(1,t) \right).
\end{align}
Consequently, by inserting inequalities \eqref{eq:11} and \eqref{eq:12} into \eqref{eq:108}, equation \eqref{eq:107} becomes:
\begin{align} \label{eq:109}
\left| \int_0^1 u_t(x,t) y(x,t) \, dx \right| & \leqslant \max \left\{ 1; \, \frac{4}{\kappa_b \, q_0^2}; \, \frac{4 C_1^2}{\kappa_r \, q_0} \right\} E(t).
\end{align}
Then, we obtain:
\begin{gather} \label{eq:110}
\left| \left[ \int_0^1 u_t(x,t) y(x,t) \, dx \right]_{t=r}^{t=T} \right| \leqslant \max \left\{ 1; \, \frac{4}{\kappa_b \, q_0^2}; \, \frac{4 C_1^2}{\kappa_r \, q_0} \right\}  \Big( E(T) + E(r) \Big).
\end{gather}
Second, applying the $\delta-$inequality, we have: 
\begin{align} \label{eq:111}
\begin{split}
& \left| \int_r^T \Big( \kappa_v u_t(1,t) + \kappa_d  u_t(1,t-\tau) \Big) y(1,t) dt + \int_r^T \kappa_a u_{xt}(1,t) y_x(1,t) dt \right| \\
& \quad \leqslant \frac{1}{\delta} \int_r^T \Big( \kappa_v^2 u_t^2(1,t) + \kappa_d^2  u_t^2(1,t-\tau) \Big) dt +  \frac{\kappa_a^2}{2\delta} \int_r^T u_{xt}^2(1,t) dt \\
& \qquad + \frac{\delta}{2} \int_r^T \Big( y^2(1,t) + y^2_x(1,t) \Big) dt.
\end{split}
\end{align}
From \eqref{eq:56} and \eqref{eq:58}, we have:
\begin{align} \label{eq:112}
\begin{split}
& y^2(1,t) + y^2_x(1,t) \leqslant \Big( \frac{2}{\kappa_r} + \frac{2}{\kappa_b} \Big) \Big( \frac{1}{q_0} u^2(1,t) + C_1^2 u^2_x(1,t) \Big).
\end{split}
\end{align}
Therefore, using\eqref{eq:40} and the preceding inequality, it follows that:
\begin{align} `\label{eq:113}
& \left| \int_r^T \Big( \kappa_v u_t(1,t) + \kappa_d  u_t(1,t-\tau) \Big) y(1,t) dt + \int_r^T \kappa_a u_{xt}(1,t) y_x(1,t) dt \right| \nonumber \\
& \quad \leqslant \frac{1}{\delta} \int_r^T \Big( \kappa_v^2 u_t^2(1,t) + \kappa_d^2  u_t^2(1,t-\tau) \Big) dt +  \frac{\kappa_a^2}{2\delta} \int_r^T u_{xt}^2(1,t) dt \nonumber \\
& \qquad \quad + \frac{\delta}{2} \int_r^T \Big( \frac{2}{\kappa_r} + \frac{2}{\kappa_b} \Big) \Big( \frac{1}{q_0} u^2(1,t) + C_1^2 u^2_x(1,t) \Big) dt  \nonumber \\
& \quad \leqslant \max \left\{ \frac{\kappa_v^2}{\delta}; \, \frac{\kappa_d^2}{\delta}; \, \frac{\kappa_a^2}{2\delta} \right\} \int_r^T \Big( u^2(1,t) + u_t^2(1,t-\tau) + u_{xt}^2(1,t) \Big) dt \nonumber\\
& \qquad \quad + \delta \left( \frac{1}{\kappa_b} + \frac{1}{\kappa_r} \right) \max \Big\{q_0^{-1}; \, C_1^2 \Big\} \int_r^T \Big( u^2(1,t) + u^2_x(1,t) \Big) dt \nonumber \\
& \quad \leqslant \delta \left( \frac{1}{\kappa_b} + \frac{1}{\kappa_r} \right) \max \Big\{q_0^{-1}; \, C_1^2 \Big\} \int_r^T \Big( u^2(1,t) + u^2_x(1,t) \Big) dt. \nonumber \\
& \qquad \quad + \frac{\max \left\{ \frac{\kappa_v^2}{\delta}; \, \frac{\kappa_d^2}{\delta}; \, \frac{\kappa_a^2}{2\delta} \right\}}{C_{\kappa_v, \kappa_d, \kappa_a}^{\gamma}} \Big( E(r) - E(T) \Big).
\end{align}

We now proceed to estimate the term $\displaystyle \int_r^T \int_0^1 \Big| u_t y_t \Big| dx \, dt$. For this purpose, by differentiating \eqref{eq:60} with respect to time, we have:
\begin{gather} \label{eq:114}
\left\{
\begin{array}{l}
\left( \sigma (y_t)_{xx} \right)_{xx} - (q (y_t)_x)_x = 0, \\
q(1) (y_t)_x(1,t) + \kappa_b y_t(1,t) - (\sigma (y_t)_{xx})_x(1,t) = u_t(1,t), \\
\sigma(1) (y_t)_{xx}(1,t) + \kappa_r (y_t)_x(1,t) = u_{xt}(1,t).
\end{array}
\right.
\end{gather}
Clearly, $y_t$ satisfies \eqref{eq:114}. Proceeding in the same manner as in \eqref{eq:112}, we get:
\begin{align} \label{eq:115}
\lVert y_t(\cdot,t) \rVert^2_{L^2(0,1)} & \leqslant \max\Big\{ \frac{2}{q_0^2}; \, \frac{2 C_1^2}{q_0} \Big\} \Big( u^2_{t}(1,t) + u^2_{xt}(1,t) \Big).
\end{align}
By applying Young's inequality, coupled with the estimates \eqref{eq:40} and \eqref{eq:115}, we thus derive the following inequalities:
\begin{align} \label{eq:116}
 \int_r^T \int_0^1 \left| u_t y_t \right| dx \, dt & \leqslant \frac{\tilde{\delta}}{2} \int_r^T \int_0^1 u_t^2 \, dx \, dt + \frac{1}{2\tilde{\delta}} \int_r^T \int_0^1 y_t^2 \, dx \, dt \nonumber \\
 & \leqslant \frac{\tilde{\delta}}{2} \int_r^T \int_0^1 u_t^2 \, dx \, dt + \dfrac{\max\Big\{ \dfrac{1}{q_0^2}; \, \dfrac{ C_1^2}{q_0} \Big\}}{\tilde{\delta}} \int_r^T \Big( u^2_t(1,t) + u^2_{xt}(1,t) \Big) \, dt \nonumber\\
 \int_r^T \int_0^1 \left| u_t y_t \right| dx \, dt  & \leqslant  \tilde{\delta} \int_r^T E(t) dt + \dfrac{\max\Big\{ \dfrac{1}{q_0^2}; \, \dfrac{ C_1^2}{q_0} \Big\}}{\tilde{\delta} C_{\kappa_v, \kappa_d, \kappa_a}^{\gamma}} \Big( E(r) - E(T) \Big),
\end{align}
for all $\tilde{\delta} > 0$. Finally, applying the triangle inequality to \eqref{eq:106} and  taking into account inequalities \eqref{eq:110}, \eqref{eq:113} and \eqref{eq:116}, we obtain:
\begin{align} \label{eq:117}
\begin{split}
& \int_r^T \left( u^2(1,t) + u_x^2(1,t) \right) \, dt \\
& \quad \leqslant \tilde{\delta} \int_r^T E(t) dt + \max \left\{ 1; \, \frac{4}{\kappa_b \, q_0^2}; \, \frac{4 C_1^2}{\kappa_r \, q_0} \right\}  \Big( E(T) + E(r) \Big) \\
& \qquad + \left( \frac{\max \left\{ \frac{\kappa_v^2}{\delta}; \, \frac{\kappa_d^2}{\delta}; \, \frac{\kappa_a^2}{2\delta} \right\}}{C_{\kappa_v, \kappa_d, \kappa_a}^{\gamma}} + \dfrac{\max\Big\{ \dfrac{1}{q_0^2}; \, \dfrac{ C_1^2}{q_0} \Big\}}{\tilde{\delta} C_{\kappa_v, \kappa_d, \kappa_a}^{\gamma}} \right) \Big( E(r) - E(T) \Big) \\
& \qquad + \delta \left( \frac{1}{\kappa_b} + \frac{1}{\kappa_r} \right) \max \Big\{q_0^{-1}; \, C_1^2 \Big\}  \int_r^T \Big( u^2(1,t) + u^2_x(1,t) \Big) dt
\end{split}
\end{align}
which leads to the expecting estimate \eqref{eq:98} by choosing $\delta$ as defined in \eqref{eq:97a}. 
\end{proof}


\begin{theorem} \label{th2}
Suppose that the function $\sigma$ is (WD) or (SD) and consider $\kappa_b, \, \kappa_r > 0$. Under the assumptions \eqref{eq:4} and \eqref{eq:24a}, the energy $E(t)$ of system \eqref{eq:1} defined in \eqref{eq:39}, exponentially decays to zero, i.e.:  
\begin{gather} \label{eq:119}
E(t) \leqslant E(0) \, e^{1 - \dfrac{t}{M_{\kappa_b,\kappa_r}}}, \quad \forall t \in \left[ M_{\kappa_b,\kappa_r}, \, +\infty \right), 
\end{gather}
where the constant $M_{\kappa_b,\kappa_r}$ is given in \eqref{eq:123} and is independent of $(u_0, \, u_1)$.
\end{theorem}

\begin{proof}
Exploiting Proposition \ref{p2} and Lemma \ref{l1} yields that \eqref{eq:40} and \eqref{eq:68} hold. Then, by using Lemma \ref{l2} we obtain inequality \eqref{eq:91}. Since $\kappa_b, \, \kappa_r$ are positive, then by virtue of Lemma \ref{l3}, \eqref{eq:98} holds. After inserting \eqref{eq:98} into \eqref{eq:91}, it follows that:
\begin{align} \label{eq:120}
\begin{split}
& \varepsilon \min\Big\{ 2 - \Upsilon_{\sigma, q}; \, 4e^{-2\tau} \Big\} \int_r^T E(t) \, dt \leqslant L(r) - L(T) \\
&  + 2 \varepsilon C_0 \left[ \tilde{\delta} \int_r^T E(t) dt + C_3 \Big( E(T) + E(r) \Big)  + C_2^\delta \Big( E(r) - E(T) \Big) \right].
\end{split}
\end{align}
where $C_0, \, C_2^\delta$ and $C_3$ are respectively defined in \eqref{eq:92} and \eqref{eq:99}. Choosing $\tilde{\delta}$ such that:
\begin{align} \label{eq:121}
	\tilde{\delta} = \dfrac{1}{4 C_0} \min\Big\{ 2 - \Upsilon_{\sigma, q}; \, 4e^{-2\tau}\Big\},
\end{align}
in the preceding expression and utilizing inequality \eqref{eq:50} in Proposition \ref{p3}, we get:
\begin{align}
\int_r^T E(t) \, dt & \leqslant 2 \varepsilon^{-1} \left( \min\Big\{ 2 - \Upsilon_{\sigma, q}; \, 4e^{-2\tau} \Big\} \right)^{-1} \Bigg[ 
\Theta_2 E(r) - \Theta_1 E(T) \nonumber \\
& \quad  + 2 \varepsilon C_0 C_3  \Big( E(T) + E(r) \Big) + 2 \varepsilon C_0 C_2^\delta \Big( E(r) - E(T) \Big)
 \Bigg] \nonumber \\
\int_r^T E(t) \, dt  & \leqslant M_{\kappa_b,\kappa_r} E(r) \label{eq:122}
\end{align}
where
\begin{align} \label{eq:123}
\begin{split}
  M_{\kappa_b,\kappa_r} := & \dfrac{2}{\varepsilon \min\Big\{ 2 - \Upsilon_{\sigma, q}; \, 4e^{-2\tau} \Big\}} \, \left( \Theta_2 + 4 \varepsilon C_0 C_3 + 2 \varepsilon C_0 C_2^\delta \right),
\end{split}
\end{align}
Finally, using Theorem 8.1 in \cite{komornik1994}, we get \eqref{eq:119}. 
\end{proof}
		
\section{Conclusion}\label{sec13}

This article addresses the well-posedness and the asymptotic behavior of the solution to a degenerate Euler-Bernoulli beam problem under the action of an axial force, subjected to a boundary control with a fixed time delay $\tau > 0$.
We established two core results under the critical condition that the gain $\kappa_v$ of the non-delayed term, strictly dominates the gain $\kappa_d$ of the delayed one, i.e., $\kappa_v > |\kappa_d|$.
First, by using the Hille-Yosida theorem, we proved that system \eqref{eq:1} admits a unique solution on a novel Hilbert state space based on weighted Sobolev spaces.
Second, we constructed a novel Lyapunov functional incorporating integral terms weighted by the rigidity, axial force, and a time delay function. By employing the energy multiplier method, we derived rigorous estimates of the Lyapunov functional and its time derivative, thereby establishing the exponential stability of system \eqref{eq:1} and providing a precise estimate of the energy decay rate. 
Consequently, the feedback control loop is shown to be sufficiently robust to preserve exponential stability, even in the presence of the axial force.
In future work, we plan to conduct a numerical analysis of the problem. This step will not only validate the preservation of the established qualitative properties but also allow us to investigate the impact of the control parameters and the time delay $\tau$ on the system's energy dissipation. 

\bibliographystyle{unsrtsiam}
\bibliography{biblio}

\end{document}